\documentclass[graybox,envcountsect,envcountsame]{svmult}

\usepackage{type1cm}        
\usepackage{makeidx}         
\usepackage{graphicx}        
\usepackage{multicol}        
\usepackage[bottom]{footmisc}

\usepackage{newtxtext}       %
\usepackage{newtxmath}       
 
  \usepackage{amsmath}
 \usepackage{amssymb}
\usepackage[curve]{xypic}
\usepackage{tikz}
\usetikzlibrary{arrows,decorations,shapes,shadows}

\newcommand{\lcm}{\operatorname{lcm}}
\newcommand{\red}{\operatorname{red}}
\newcommand{\tot}{\operatorname{tot}}
\newcommand{\specan}{\operatorname{specan}}
\newcommand{\supp}{\operatorname{supp}}
\newcommand{\bl}{\operatorname{Bl}}
\newcommand{\proj}{\operatorname{Proj}}

\makeindex             

\begin{document}

\title*{Lipschitz fractions of a complex analytic algebra and Zariski saturation }
\author{Fr\'ed\'eric Pham and Bernard Teissier}

%
%
\maketitle
\begin{abstract}\noindent
This text is the English translation, due to Naoufal Bouchareb, of an unpublished manuscript of 1969 (the French version is available on HAL as hal-00384928) inspired by Zariski's theory of saturation. Its publication is justified by the fact that it appears now as a precursor of the recently developed study of the Lipschitz geometry (for the outer metric) of germs of complex analytic spaces. This version contains some new footnotes and an additional bibliography.
\end{abstract}
\section*{Introduction}

While seeking to define a good notion of equisingularity (see \cite{Zar65-a}, \cite{Zar65-b}) , Zariski was led to define in \cite{Zar68} what he calls the {\it saturation} of a local ring:\index{saturation} the saturated ring $\widetilde A$ of a ring $A$ contains $A$ and is contained in its normalization $\overline{A}$, and for a complete integral ring of dimension $1$, the datum of the saturated ring is equivalent to the datum of the set of  Puiseux characteristic exponents of the corresponding algebroid curve.

In the case of  {complex analytic algebras}, it is well known that the normalization $\overline{A}$ coincides with the set of germs of meromorphic functions with bounded module; among the intermediate algebras between $A$ and $\overline{A}$, there is one which  can be introduced quite naturally: it is the algebra of the germs of Lipschitz meromorphic functions. We propose to study this algebra, first formally (Section \ref{section1}), then geometrically (Section \ref{section2}), and to prove (Sections \ref{section3} and \ref{section5}) that at least in the case of hypersurfaces, it {coincides with the Zariski saturation}. In Section \ref{section4},  in the case of a reduced but not necessarily irreducible curve, we show how the constructions of Sections \ref{section1} and  \ref{section2} provide {a sequence of rational exponents} (defined intrinsically, without reference to any coordinates system), which  {generalizes the}  {sequence of  characteristic  Puiseux  exponents \index{Puiseux exponent!characteristic} of an irreducible curve}.
Finally, in Section \ref{section6}, we recover in a  very simple way  the result of Zariski which states that the  {equisaturation} of a family of hypersurfaces implies their  {topological equisingularity} (we even obtain the  {Lipschitz equisingularity}, realized by a Lipschitz deformation of the ambient space).

All the arguments are based on the techniques of  {normalized blow-ups} \index{blow-up!normalized} (recalled in the Preliminary Section; see also \cite{Hir64-b}), and we thank Professor H. Hironaka who taught it to us.\footnote{We are also very grateful to Mr. Naoufal Bouchareb who brilliantly and expertly translated our 1969 manuscript from French to English and from typewriting to LaTeX.} 

\section*{Preliminaries}

\noindent{\bf  (Reminders on the techniques of normalized blow-ups and majorations of analytic functions)}

\subsection*{Conventions}
In what  follows the rings are commutative, unitary and noetherian. A ring $A$ is said to be {\it normal} \index{normal!ring} if it is integrally closed in its total ring  of fractions $\tot(A)$. An analytic space $(X,\mathcal{O}_X)$ is said to be  {\it normal} if at every point $ x\in X $, $\mathcal{O}_{X,x}$ is normal. \index{normal!analytic space} We will denote by $\overline{A}$ the integral closure of a  ring  $A$ in $\tot(A)$.

\subsection{Universal property of the normalisation}
Let $n : \overline{X} \rightarrow X$ be the {\it normalisation} \index{normalization!universal property} of an analytic space $\overline{X}$, i.e.,  
 $\overline{X}=\specan_X\overline{\mathcal{O}_X}$, where $\overline{\mathcal{O}_X}$ is a finite $\mathcal{O}_X$-algebra satisfying  $(\overline{\mathcal{O}_X})_x = \overline{\mathcal{O}_{X,x}}$.
\begin{definition} 
For every normal analytic space $Y\overset{f}{\longrightarrow}X$ above $X$ such that the $f$-image of any irreducible component of $Y$ is not contained in $N=\supp \ \overline{\mathcal{O}_X} / \mathcal{O}_X $ (the analytic subspace of points in $X$ where $\mathcal{O}_{X,x}$ is not normal), there exists a unique factorization:
  
  $$ \xymatrix{
    Y \ar[rr]^{\overline{f}} \ar[rrd]_f && \overline{X} \ar[d]^n \\
     && X
  }
  $$
\end{definition}

\begin{proof}

\begin{description}
    \item[(a)] Algebraic version:
\\
 Let  $\varphi : A  \rightarrow  B $ be a homomorphism of rings. Let $(\mathfrak{p}_i)_{i=1,\ldots,k}$ be the prime ideals of $0$ in $B$.
\\
We suppose that:
\begin{enumerate}
    \item[(i)] $B$ is normal;
    \item[(ii)]  for  every $i=1,\dots,k$, $C_{\overline{A}}(A)$ is not included in $   \varphi^{-1}(\mathfrak{p}_i)$,  where $ C_{\overline{A}} (A)$ denotes the conductor of $A$ in $\overline{A}$:
   \[
       C_{\overline{A}}(A) =  \{ g \in  A/ g\overline{A} \subset A \}.
   \]
\end{enumerate}
\noindent
Then, there is a unique factorization:
 
$$ \xymatrix{&\overline{A} \ar[rd]^{\overline{\varphi}} & \\ A    \ar[ru]  \ar[rr]^\varphi && B    } $$

Indeed, by the Prime Avoidance Lemma (see \cite[\S 1]{Bou61}), there exists $g\in C_{\overline{A}}(A)$ such that $g \notin \varphi^{-1}(\mathfrak{p}_i)$ for all $i=1,\dots,k $. This implies  that $\varphi(g)$ is not a divisor of $0$ in $B$. For every  $h \in \overline{A}$, set:
$$
\overline{\varphi}(h)=\frac{\varphi(g.h)}{\varphi(g)} \in \tot (B)
$$
Since $h$ is integral on $A$, $\overline{\varphi}(h)$ is integral on $\varphi(A)$, and thus also on $B$. Hence, $\overline{\varphi}(h) \in B$ and $\overline{\varphi}$ is the desired factorization. The uniqueness is obvious.
\\
\item[(b)] Geometric version:

Let  ${Y} \stackrel{{f}}{\longrightarrow} {X}$  satisfy the conditions of the statement. The conditions of the statement remain true locally at $y \in {Y}$ since if $\varphi^{-1}(N)$ contains locally an irreducible component of $Y$, it contains it globally. We deduce from this  that the local homomorphism:
\
\begin{center}
    ${\mathcal{O}}_{X, f(y)} \longrightarrow \mathcal{O}_{Y, y}$
\end{center}
satisfies the conditions of the algebraic version.
We then have the unique factorization:
$$\xymatrix {
  &  \overline{\mathcal{O}_{X, f(y)}} \ar[rd]& \\ \mathcal{O}_{X, f(y)} \ar[rr] \ar[ru] && \mathcal{O}_{Y,y}} $$
and by the coherence of $\overline{\mathcal{O}}_X$, the existence and uniqueness of the   searched morphism.
 \end{description}
\end{proof}
\subsection{Universal property of the blowing-up (see \cite{Hir64-a})}  \label{subsec:Universal property of the blow-up} \index{blow-up!universal property}

\begin{proposition}  Let  $Y \hookrightarrow X$ be two analytic spaces and let $\cal I$ be the ideal of $Y$ in $X$. There exists a unique analytic space $ Z \stackrel\pi{\longrightarrow} X$ over $X$ such that:

    \begin{enumerate}
        \item[i)] $\pi^{-1}(Y)$ is a divisor of $Z$, i.e., ${\cal I}. \mathcal{O}_Z$ is invertible.
        \item[ii)] for every morphism $T\stackrel{\varphi}{\longrightarrow} X$ such that ${\cal I}.\mathcal{O}_T$ is invertible, there is a unique factorization:
$$ \xymatrix {
    T \ar[rr]^{\bl(\varphi)} \ar[drr]^{\varphi} && Z \ar[d]^\pi \\
   && X }$$
    \end{enumerate}
The morphism $Z\stackrel{\pi}{\rightarrow} X$  is called the {\it blow-up} of $X$ \index{blow-up} along $Y$. Recall that $\pi$ is bimeromorphic, proper and surjective and that $\pi\vert Z\setminus \pi^{-1}(Y)$ is an isomorphism on $X \setminus Y$.
 \end{proposition}
 
\subsection{Universal property of the normalized blow-up}  \label{subsec:P3}  \index{blow-up!normalized} 

\begin{proposition} 
 Let $Y  \hookrightarrow  X$ such that $X$ is normal outside of $Y$. Then, for every morphism $T\stackrel\varphi{\longrightarrow} X$ such that:
\begin{enumerate}
    \item[i)] $T$ is normal;
    \item[ii)] ${\cal I}.{\cal O}_T$ is invertible,
\end{enumerate}
 there exists a unique factorization.
 $$\xymatrix {
    T \ar[rr]^{\overline{\bl(\varphi)}} \ar[drr]^{\varphi} && \overline{Z} \ar[d]^{n\circ \pi} \\
   && X }$$
  \end{proposition}
 
\begin{proof}
It is sufficient to check that the factorization $T\stackrel{\bl(\varphi)}{\longrightarrow} {Z}$ satisfies the conditions of Subsection  \ref{subsec:Universal property of the blow-up}. Since $\pi\vert
Z \setminus \pi^{-1}(Y)$ is an isomorphism, $Z \setminus \pi^{-1}(Y)$ is normal and it is sufficient to verify that the image of each irreducible component of $T$ meets $Z \setminus\pi^{-1}(Y)$. But the inverse image of $\pi^{-1}(Y)$ by $\bl(\varphi)$ is a divisor by assumption. Since $T$ is normal, this divisor cannot  contain any irreducible component.
\end{proof}

\subsection{Normalized blow-up and integral closure of an ideal} \label{subsec:P4}
\noindent (See also \cite[Chap. II]{Lip69}.)
 \vskip0,3cm
Let $A$ be the analytic algebra of an analytic space germ $(X,0)$, let  $I$ be an ideal of $A$ and let  $Y  \hookrightarrow  X$ be the corresponding sub-germ. It is known that the blow-up of the germ $Y$ in the germ $X$\footnote{Here, as in other places, we abuse language to identify the germ $(X,0)$ with one of its representatives.} is the projective object $Z=\proj_{A}E$ over $X$ associated with the graded algebra $E=\underset{n \geqslant 0}{\oplus}{I}^n$.
The normalization of $Z$ can be written $\overline{Z}=\proj_{A}\overline{E}$, with ${\overline{E}}=\underset{n \geqslant 0}{\oplus} \overline{I^n}$
(where, for an ideal  $J$ of $A$, we  define:
$$
\overline{J}=\left\{h \in {\tot}(A) \ \  | \ \ \exists j_{1} \in J, {j}_{2} \in {J}^{2}, \ldots, {j}_{{k}} \in {J}^{{k}} \colon \ \ {h}^{{k}}+{j}_{1} {h}^{{k}-1}+\cdots+{j}_{{k}}=0\right\}, 
$$
which is the ideal of $\overline{A}$ called the  {\it integral closure of the ideal $J$ in $\overline{A}$}). \index{integral closure!of a ring}

As an object over $\overline{X}$, the space $\overline{Z}$ equals $\proj_{\overline{A}}\overline{E}$. But since $\overline{E}$ is a  graded $\overline{A}$-algebra of finite type, there exists a positive integer $s$ such that the graded algebra $${\overline{E}^{(s)}}=\underset{n \geqslant 0}{\oplus} {\overline{I^{n.s}}}$$ is generated   by its degree $1$ elements:
${\overline{E}_{1}^{(s)}}={\overline{I^{s}}}$. 
But then, $\overline{E}_{n}^{(s)}=(\overline{I^{s}})^n$, and as we know that there is a canonical isomorphism $\overline{Z}=\proj_{\overline{A}}\overline{E}^{(s)}$, we see that the normalized blow-up $\overline{Z}$ of $I$ in $A$, with  its canonical morphism to $\overline{X}$, coincides with 
  the blow-up of $\overline{I^{s}}$ in $\overline{A}$.
 
\begin{proposition} \label{prop:preliminaries i}
$I$ and $\overline{I}$ generate the same ideal of $\mathcal{O}_{\overline{Z}}$, i.e., $I\mathcal{O}_{\overline{Z}}=\overline{I}\mathcal{O}_{\overline{Z}}$.
  \end{proposition}
  
\begin{proof}
$\overline{E}$ is a finite type $E$-module, so for $N$ big enough, $I.\overline{I^{N}}=\overline{I^{N+1}}$. But $\overline{I}.\overline{I^N}\subset \overline{I^{N+1}}$, therefore:
\begin{equation}  \label{equation} \overline{{I}}{{\cal O}}_{\overline{Z}} \ . \   \overline{{I}^{{N}}} {{\cal O}}_{\overline{{Z}}}\subset{I} {\cal O}_{\overline{{Z}}}  \ . \  \overline{{I}^{{N}}} {{\cal O}}_{\overline{{Z}}}.
\end{equation}
 But if $N=k.s$, then $\overline{I^{N}}. \ {\cal O}_{\overline{Z}}=(\overline{I^{s}})^k. \  {\cal O}_{\overline{Z}}$. The latter ideal being invertible, we can simplify by $\overline{I^{N}}{\cal O}_{\overline{Z}}$ in the inclusion \eqref{equation}. Then $\overline{I}. {\cal O}_{\overline{Z}}\subset{I}.{\cal O}_{\overline{Z}}$.  The reverse inclusion is obvious.
\end{proof}

 \begin{proposition} \label{prop:preliminaries ii}
 $\overline{I}$  coincides with  the  set of elements of $\overline{A}$ which  define a section of $I.{\cal O}_{\overline{Z}}$.
 \end{proposition}
 
 \begin{proof}
 If $f\in \overline{I}$, then $f$ obviously defines a section of $\overline{I}\mathcal{O}_{\overline{Z}}$. But $\overline{I}\mathcal{O}_{\overline{Z}}=I\mathcal{O}_{\overline{Z}}$ according to Proposition \ref{prop:preliminaries i}. Conversely, suppose that $f\in\overline{A}$ defines a section of $I\mathcal{O}_{\overline{Z}}$; by writing what this means in some affine open sets $\overline{Z}_{(g_k)}\subset \overline{Z}$, where $g_k\in\overline{I^s}$, one finds that there must exist some integers $\mu_k$ such that $f  .  g_{k}^{\mu_k} \in I .( \overline{I^{s}})^{\mu_k}$. 

 Let  $(g_k)$ be a finite family of generators of $\overline{I^{s}}$. For $N$ large enough, every monomial of degree $N$ in the $g_{k}$'s will contain one of the $g_{k}^{\mu_k}$ as a  factor, so:
 $$f. \ (\overline{I^s})^N\subset I. \ (\overline{I^s})^N,$$
i. e.,  by choosing a base $(e_i)$ of $(\overline{I^s})^N$,
$$f.\ e_i=\sum_{j} a_{ij}e_j,\hspace{0,5cm} a_{ij}\in I.$$
Since $\overline{A}$ can be supposed to be integral, we deduce from this that

$${\rm det}(f. \  {1}-\left \| a_{ij} \right \|)=0,$$ 
which is  an  equation of  integral dependence for $f$ on $I$.
\end{proof}

\subsection{Majoration theorems} \label{subsec:P5}
\begin{theorem}  [well known, see for example \cite{Abh64}] \label{thm:majoration 1}
Let $A$ be a reduced complex analytic algebra and let $(X,0)$ be  the associated germ. For every  $h \in \tot(A)$, the    following properties are equivalent:
\begin{enumerate}
    \item[i)] $h\in \overline{A}$
    \item[ii)] $h$ defines on $X^{\red}$ a  function germ  with bounded module.
    
\end{enumerate}

\end{theorem}
 
 \begin{theorem} \label{thm:majoration 2}
 Let $A$ be a complex analytic algebra,  let $(X,0)$ be  the associated germ, let  $I=(x_1,\dots,x_p)$ be an ideal of $A$ and let $\overline{Z}$  be the normalized blow-up of $I$ in $X$. For every $h \in \tot(A)$, the  following  properties are equivalent:
 \begin{enumerate}
     \item [i)] $h\in I.  \mathcal{O}_{\overline{Z}}$
     \item[ii)] $h$ defines on $X^{\red}$ a germ of function with module   bounded by $sup|x_i|$  (up to multiplication by a constant).
 \end{enumerate}
 
 \end{theorem}
 \begin{proof}  Let $A$  be a noetherian local  ring and let $I=(x_1,\dots,x_p)$ be a  principal ideal of $A$. Then $I$ is generated by one of the $x_i$'s (easy consequence of Nakayama's lemma).  
 Thus $\overline{Z}$ is covered by a finite number of open sets such that  in each of them,  one of the $x_i$'s generates $  I.\mathcal{O}_{\overline{Z}}$.

 To show that $\frac{|h|}{\sup |x_i|}$ is bounded on $X$, we just have to prove  that it is bounded on each of these open-sets, since $\overline{Z}\rightarrow X$ is proper and surjective.
In the open set where $x_i$ generates $  I.\mathcal{O}_{\overline{Z}}$,  $\frac{|h|}{\sup |x_i|}$ is  bounded if and only if   $\frac{|h|}{|x_i|}$ is bounded and we are back to theorem \ref{thm:majoration 1}.
 
 \end{proof}
\begin{corollary}  (from Preliminary \ref{subsec:P4}) \label{cor:CorollaryM2} 
 For every  $h\in \overline{A}$, the following properties are equivalent:
 \begin{enumerate}
     \item [i)] $h\in \overline{I} $
     \item[ii)] $h$ defines on $X^{\red}$ a germ of function  with  module bounded by $sup|x_i|$   (up to multiplication by a constant).
 \end{enumerate}
 \end{corollary}

 \section{Algebraic characterization of Lipschitz fractions}
\label{section1} 
 Let  $A$ be a reduced complex analytic algebra and  let $\overline{A}$ be its normalization 
  ($\overline{A}$ is a direct sum of normal analytic algebras, each being therefore an integral domain, one per irreducible component of the germ associated to $A$). Consider the ideal:
$$I_A=\ker(\overline{A}  \ \underset{\mathbf C}{\widearc{\otimes}} \ \overline{A}\rightarrow \overline{A}\underset{A}{\otimes }\overline{A}),$$
where   $\widearc{\otimes}$ means the operation on the algebras that corresponds to the cartesian product of the analytic spaces.

 \begin{definition} 
 We will  call {\it Lipschitz saturation} \index{saturation!Lipschitz} of $A$ the algebra:
\[
\widetilde{A}=\{ f\in\overline{A} \ \ | \ \  f{\widearc{\otimes}} 1  -  1{\widearc{\otimes}} f \in\overline{I_A}\}
\]
where $\overline{I_A}$ denotes the integral closure  of the ideal $I_A$ (in the sense of Subsection  \ref{subsec:P4}).
 \end{definition}
\begin{theorem} \label{thm:Theorem0} $\widetilde{A}$  is the set of fractions of $A$ that define Lipschitz function germs on the analytic space $X$, a small enough representative of the germ $(X,0)$ associated to $A$.
\end{theorem}
 \begin{proof}
 
Firstly, let us remark that  that every Lipschitz function is locally bounded and that  the set of bounded fractions of $A$ constitutes the normalization $\overline{A}$ (Theorem  \ref{thm:majoration 1}). 
However, denoting by $\overline{X}$  the disjoint sum of germs of normal analytic spaces associated to the algebra $\overline{A}$, the Lipschitz condition 
$ |f(x)-f(x')|\leqslant C \sup|z_i-z'_i| $
for an element $f\in \overline{A}$ is equivalent to say that on $\overline{X}\times\overline{X}$, the function $f\widearc{\otimes} 1  -  1\widearc{\otimes} f $ has its  module  bounded   by the supremum of the modules of the $z_i\widearc{\otimes}1-1\widearc{\otimes} z_i$, where $z_1, \ldots,z_r$ denotes a system of generators of the maximal ideal of $A$. But the ideal generated by $z_i\widearc{\otimes}1-1\widearc{\otimes} z_i, i=1 \ldots,r$  is nothing but  the ideal $I_A$ defined above. Theorem \ref{thm:Theorem0} is therefore a simple application of Corollary \ref{cor:CorollaryM2}.
 \end{proof}

\begin{corollary}   \label{cor:Corollary0} $\widetilde{A}$ is a local algebra (and thus an analytic algebra).
\end{corollary}

\begin{proof}
Since the algebra $\widetilde{A}$ is  intermediate between A and $\overline{A}$, it is a direct sum of analytic algebras.  If this sum had more than one term, the element $1\oplus 0\oplus\dots\oplus0$ of $\widetilde{A}$ would define on $X$ a germ of  function equal to $1$  on at least one of the irreducible components of $X$, and  to $0$ on another of these components. But such a function could not be continuous on $X$ and  a fortiori not Lipschitz. 
\end{proof}

The following geometric construction, which comes from  Subsection \ref{subsec:P4}, will play a fundamental role in the sequel. 
 We will associate the  following commutative diagram to the analytic space germ $X$: 

 $$\xymatrix { D_X  \ \    \ar[d]  \ar@{^{(}->}[r]&\ \ E_X \ar[d] \\
 \overline{X} \underset{X}{\times}  \overline{X} \ \ \ar@{^{(}->}[r]& \ \   \overline{X} \times   \overline{X}
 }$$
where $E$ denotes the projective object over $\overline{X}\times\overline{X}$ obtained by   the blow-up with  center $ \overline{X} \underset{X}{\times}  \overline{X}$ followed by the normalization  (i.e.,  $E_X$ is the normalized blow-up of the ideal $I_A$ which defines $ \overline{X} \underset{X}{\times}  \overline{X}$  in $\overline{X}\times\overline{X}$); 
the space $D_X$ is the {\it exceptional divisor}, \index{exceptional divisor} inverse image of $ \overline{X} \underset{X}{\times}  \overline{X}$ in $E_X$.  According to Subsection \ref{subsec:P4},  the condition:
$$f{\widearc{\otimes}} 1  -  1{\widearc{\otimes}} f \in\overline{I_A},$$
which defines $\widetilde{A}$, is equivalent to:
$$(f{\widearc{\otimes}} 1  -  1{\widearc{\otimes}} f)|  D_X=0.$$
In other words, the germ  $\widetilde{X}$ associated with the analytic algebra $\widetilde{A}$ is nothing but 
the coequalizer\footnote{So we have a canonical morphism of analytic spaces ${D}_{X} \rightarrow  \widetilde{{X}}.$} of the  canonical  double arrow
$${D}_{{X}} \rightrightarrows \overline{{X}}$$ obtained by composing the natural map ${D}_{{X}} \rightarrow \overline{X}\times\overline{X}$ with the two projections to $\overline{{X}}$.
This germ of analytic space $\widetilde{X}$ will be called the {\it the  Lipschitz saturation} \index{saturation!Lipschitz} of the germ $X$.

It  is easy to see that the above local  construction can be  globalized: it is well known for the  objects $E_{X}$ and $D_{X}$, which come from blow-ups and normalizations. Likewise for $\widetilde{X}$:   it is easy to define, on an analytic space $X = (| X |, {\cal O}_{X})$, the sheaf     $\widetilde{\cal O}_{X}$  of germs  of Lipschitz fractions, and to verify that it is a coherent sheaf  of ${\cal O}_{X}$-modules (as a subsheaf of the coherent sheaf $\overline{\cal O}_{X}$ ); we thus define an analytic space $\widetilde{X} = (| X |, \widetilde{\cal O}_{X})$ called {\it the Lipschitz saturation} of $X = (| X |, {\cal O}_{X})$, whose underlying topological space $| X |$ coincides with that of $X$ (in fact, the canonical morphism is bimeromorphic and with Lipschitz inverse, so it is a homeomorphism).

\vskip0,3cm\noindent
{\bf Question 1.} The  inclusion $\widetilde{A} \subset \overline{A}$ was obvious in the transcendental interpretation: "every Lipschitz fraction is bounded". 

But if one is interested in objects other than analytic algebras, for example in algebras of formal series, there is no longer any reason for $\overline{A}$ to play a particular role in the definition of $\widetilde{A}$. For example, we can define, for any extension $B$ of $A$ in its total fractions ring, the {\it Lipschitz saturation of $A$ in $B$}:  \index{saturation!Lipschitz}  
$$\widetilde{A}(B)=\left\{f \in B | f\otimes 1-1 \otimes f \in \overline{I_{A(B)}}\right\}$$
with  $$I_{A(B)}=\ker(B \underset{\mathbf C}{\otimes} B \rightarrow B  \underset{A}{\otimes} B).$$

The question then arises whether we still have the inclusion $\widetilde{A}(B) \subset \overline{A}$.

\section{Geometric interpretation of the exceptional divisor $\mathbf{{D}_{X}}$: pairs of  infinitely near points on $X$} \label{section2}

Each point of ${D}_{{X}}^{\red}$ (the reduced space of the exceptional divisor ${D}_{{X}}$) will be interpreted as a  {pair of infinitely near points}  \index{infinitely near points} on X. The different irreducible components $^{\tau} {D}_{{X}}^{\red}$ of ${D}_{{X}}^{\red}$, labelled by the index $\tau$, will correspond to different {\it types} of  infinitely near points. The image of $^{\tau} {D}_{{X}}^{\red}$ in X   (by the canonical map $^{\tau} {D}_{{X}}^{\red} \hookrightarrow  {D}_{{X}} \rightarrow \widetilde{{X}} \rightarrow {X}$) is an irreducible analytic subset germ $^{\tau} {X} \subset {X}$, which we can call {\it confluence locus  \index{confluence locus} of  the infinitely near points of type $\tau$}. Among the  types of infinitely near points, it is necessary to distinguish the {\it trivial types} \index{confluence locus!trivial type} whose confluence points are the irreducible components of X:  the  {\it generic} point of a {\it trivial}  $^{\tau} {D}_{{X}}^{\red}$  will be a pair obtained  by making two points of X tend towards the same {smooth} point of X. All the other ({\it non-trivial}) types have their confluence locus  consisting of   {singular} points of X: for example, we will see later that every hypersurface has as non-trivial confluence locus the components of codimension $1$ of its singular locus.

 What do the Lipschitz fractions become in this context?  
 We have seen in Section \ref{section1} that a Lipschitz fraction is an element $ f \in {\overline{A}}$ such that $\left( f \otimes 1 - 1 \otimes f\right) | {D}_{X}=0$. But, since ${D}_{X} $ is a divisor of the normal space ${E}_X$, this condition will be satisfied everywhere if it is only satisfied  in a neighbourhood  of a point of each irreducible component of this divisor; or, in intuitive language: ``to verify the  Lipschitz condition  it is enough to verify  it for a pair of infinitely near  points of each type''.  Notice that we do not need to worry about trivial types, for which the condition is trivially satisfied for all $ f \in {\overline{A}}$ (note also that the trivial $^{\tau} {D}_{{X}}$  are reduced). 

We deduce  from this the following result.
\begin{theorem} \label{theorem 1}
A meromorphic function  which is locally bounded on the complex analytic space $X$ is locally Lipschitz at every  point if only it is locally Lipschitz at one point in each confluence locus $ ^{\tau} {X} $.
\end{theorem} 

To give a first  (very rough) idea of the shape of the $^{\tau} {D}_{{X}}^{\red}$, let us look at their images in the space $ \widehat { {E}}_{X} $ defined by blowing-up the ideal $  {I}_{A} $ in $\overline{{X}} \times \overline{{X}}$. The space  ${E}_{X}$  that we are interested in is the normalization of $ \widehat { E}_{X} $. But $ \widehat { E}_{X} $ has a simpler geometric interpretation: it is the closure in $\overline{{X}} \times \overline{{X}} \times \mathbf{P}^{{N-1}}$ of the graph  $ \Gamma$ of the map
$$
(\overline{X} \times \overline{X}-\overline{X}  \underset{X}{\times} \overline{X}) \longrightarrow \mathbf{P}^{N-1}
$$
which maps  each pair $(\overline{x} , \overline{x}')$ outside of the diagonal to the line defined, in homogeneous coordinates, by: 
$$ (z_{1}-z'_{1} : z_{2}-z'_{2} : ... : z_{N}-z'_{N})\ ,$$
where $(z_{1},z_{2},\ldots,z_{N})$ denotes a system of generators of the maximal ideal of ${\mathcal O}_{X,x}$.

We will denote by  $\hat{z}   \colon  \widehat { E}_{X} \rightarrow \mathbf{P}^{{N-1}}$ 
the  underlying morphism  and  by $  \hat{z}'   \colon  \widehat {D}_{X} \rightarrow {\mathbf P}^{N-1}$ the restriction of $\hat{z}$ over $\overline{{X}} \underset{X}{\times} \overline{{X}}$ (these morphisms depend on the choice of the generators $(z_{1},z_{2},\ldots,z_{N})$). The fiber  $\widehat{ D}_{X}(x)$  of the exceptional divisor $\widehat{D}_{X}$ over a point ${x} \in {X}$ is the disjoint sum of a finite number of algebraic varieties (as many as $\overline{{X}}  \underset{X}{\times} \overline{{X}}$ has points over x) that are embedded in $\mathbf{P}^{{N-1}}$ by the map $ \hat{z}'{ | \widehat { D}_{X}(x)}$.   In particular, if $x$ is a smooth point, $\widehat{ D}_{X}(x)$ is nothing but the projective space $\mathbf{P}^{n-1}$ associated with the  tangent space  to $X$ at $x$.

By composition with the finite morphisms ${E}_{X}  \rightarrow \widehat{ E}_{X}$ (normalization) and  ${D}_{X} \rightarrow \widehat{ D}_{X}$, we deduce from $\hat{z}$ and $\hat{z}'$ two morphisms 
$$ \tilde{{z}} \colon {E}_{X} \rightarrow \mathbf{P}^{N-1}$$
$$ \tilde{{z}}' = \tilde{{z}}{ | \ {D}_{X}} \colon  {D}_{X} \rightarrow \mathbf{P}^{N-1}$$
where $\tilde{{z}}'$ has the following property: its restriction to the  fiber ${D}_{X}(x)$ of ${D}_{X}$ over ${x} \in {X}$ is a finite morphism.

\begin{corollary}
If $ X \subset \mathbf{C}^{N}$ is of pure dimension $n$, the  confluence loci  ${}^{\tau}X$ are of dimension at least equal to $2n-N$.
\end{corollary}
\begin{proof} According to the finiteness of the   above morphism, $\dim{{D}_{X}(x)} \leqslant N-1$,  so each irreducible component  $^{\tau} {D}_{{X}}^{\red}$ of ${D}_{X}$ will have an image $^{\tau} {{X}}$ in $X$ of dimension:
$$
\dim {}^{\tau}X \geq \dim {}^{\tau} D_{X}^{\red}-(N-1) = (2n-1)-(N-1) = 2n-N .
$$
\end{proof}
\noindent
{\bf The special case of hypersurfaces.}  In this case, $ N = n+1$, so the confluence loci are of dimension at least equal to $n - 1$. The only non-trivial confluence loci are the codimension $1$ components of the singular locus of $X$. Furthermore, the  fibres $^{\tau}{D}_{X}(x)$ of the non-trivial  $^{\tau}{D}_{X}$  are sent onto $\mathbf P^{N-1}$ by finite morphisms (which are surjective by a dimension argument).
 In the special case of hypersurfaces, Theorem  \ref{theorem 1} is thus formulated as follows:
\begin{theorem}  \label{theorem 1H} 
A meromorphic function on a complex analytic hypersurface $X$ is locally Lipschitz at every point if only it is   locally Lipschitz   at one point of each irreducible component (of codimension $1$) of its polar locus.  
\end{theorem}

 \begin{definition}  At a generic point of the divisor ${}^{\tau} D_{X}^{\red}$, this divisor is a smooth divisor of the smooth space ${E}_{X}$. Let $s$ be its irreducible local equation. The ideal of the non reduced divisor ${}^{\tau} {D}_{X}$ is then locally of the form (${s}^{\mu(\tau)}$), where $\mu(\tau)$ is a positive integer, the {\it multiplicity of the divisor} $^{\tau} {D}_{X}$". \index{multiplicity!of a divisor}
\end{definition}
\section{Lipschitz fractions relative to a parametrization} \label{section3}

Let $R \subset A$ be an analytic subalgebra of $A$ and  let $S$ be  the associated analytic space germ. By considering $X$ as a relative analytic space over $S$, we are going to proceed to a construction analogous to that of 
Section \ref{section1}, where the product $\bar{{X}} \times \bar{{X}} $ is replaced by the fiber product on $S$. This gives a diagram:
$$ \xymatrix{
    D_{X/S} \ \   \ar[d]   \ar@{^{(}->}[r]& \ \ E_{X/S} \ar[d] \\
    \overline{X} \underset{X}{\times}   \overline{X} \ \   \ar@{^{(}->}[r] & \ \   \overline{X} \underset{S}{\times}   \overline{X}
  }
  $$
which enables one to define the {\it algebra of Lipschitz fractions relative to S}:  \index{Lipschitz fraction}
$$
\widetilde{A}^{{R}}=\left\{{f} \in \overline{{A}} \ | \ ({f}{\widearc{\otimes}} 1-1{\widearc{\otimes}} {f}){| D_{{X} / {S}}}=0\right\},
$$
whose geometric interpretation is given by the ``relative'' analog to Theorem \ref{thm:Theorem0}:
\begin{theorem}[Relative Theorem \ref{thm:Theorem0}]
$\widetilde{A}^{{R}}$ is the set of fractions of A that satisfy a Lipschitz condition: 
$$
\left|f(x)-f\left(x^{\prime}\right)\right| \leq {C} \sup _{i}\left|z_{i}-z_{i}^{\prime}\right|
$$
 for every pair of points $(x, x')$ taken in the same fiber of ${X} / {S}$ (with the same constant ${C}$ for all fibers). 
\end{theorem}

Notice the inclusion $\widetilde{A} \subset \widetilde{A}^{R}$, which is evident in the geometric interpretation. Formally, this inclusion can also be deduced from the existence of a ``morphism'' from the above  relative diagram  to the absolute diagram of Section \ref{section1}:

 $$\xymatrix {   & D_X  \ar[dd]  \ \  \ar@{^{(}->}[rr] & & \ \  E_X   \ar[dd] \\
 D_{X/S}  \ar@{.>}[ru]   \ar[rd]    \ar@{^{(}-}[r]  & \ar[r] & E_{X/S}  \ar[d]  \ar@{-->}[ru]  & \\
 &  \overline{X} \underset{X}{\times}  \overline{X}  \ \ar@{^{(}->}[r]&  \  \overline{X} \underset{S}{\times}  \overline{X} \ \ar@{^{(}->}[r]& \  \overline{X} {\times}  \overline{X}
 }$$
where the dotted arrow $\xymatrix { \ar@{-->}[rr]& &}$ is defined by the universal property of the normalized blow-up (see Subsection \ref{subsec:P3}, noting that $\overline{X} {\times}  \overline{X}$ is normal).

We will now assume that $X$ is of pure dimension  $n$ and we will be interested in the case where $R$ is a {\it parametrization} \index{parametrization} of $A$, i.e.,  the regular algebra $\mathbf{C}\left \{ {z_{1}, z_{2}, ..., z_{n}} \right \}$ generated by  a system of parameters of $A$ (an $n$-uple of elements of $A$ such that the ideal generated in $A$ contains a power of the maximal ideal). In other words, $X \rightarrow S$ is a finite morphism from ${X}$ to a Euclidean space of dimension equal to that of $X$. Let ${z}=(z_{1}, z_{2}, ..., z_{N})$ be a system of generators of the maximal ideal of A, and let us consider  $n$ linear combinations of them: 
$$
\begin{array}{l}{(a z)_{1}=a_{11} z_{1}+a_{12} z_{2}+\cdots+a_{1 N} z_{N}} \\ {(a z)_{2}=a_{21} z_{1}+a_{22} z_{2}+\cdots+a_{2 N} z_{N}} \\ {(a z)_{n}=a_{n 1} z_{1}+a_{n 2} z_{2}+\cdots+a_{n N} z_{N}} \end{array}$$ $${\left(a_{i, j} \in \mathbf{C}\right)}$$

The set of the $a=(a_{ij})$ for which $\mathbf{C}\left \{ {(az)_{1}, (az)_{2}, ..., (az)_{n}} \right \}$ is a parametrization of  ${A}$ forms, obviously, a dense open  set of the space $M_{N \times n} (\mathbf C)$ of all the $N \times n$ matrices . We will say more generally that a family $\cal P$ of parametrizations is {\it generic} \index{generic!family of parametrizations} if for every system ${z}=(z_{1}, z_{2}, ..., z_{n})$ of generators of the maximal ideal of $A$, the set of matrices $a$  for which $\mathbf{C}\left \{ {(az)_{1}, (az)_{2}, ..., (az)_{n}} \right \} \in {\cal P} $ contains a dense open set of $M_{N \times n} (\mathbf C)$. 

We propose to prove the:
\begin{theorem} \label{theorem2}
For any generic  family  ${\cal P} $ of parametrizations,
$$
{\widetilde{A}=\bigcap_{R \in {\cal P}} \widetilde{A}^{R}}
$$
\end{theorem}
It follows from this theorem that the following two questions admit identical answers: 
\vskip0,2cm
\noindent
 {\bf Question 2.} Is  the equality $\widetilde{A}= \widetilde{A}^{R}$   generically true (i.e., for a generic family  $R$ of parametrizations)? 
\vskip0,2cm
\noindent
 {\bf Question 2'.} Is $\widetilde{A}^{R}$  generically independent of $R$? 
\vskip0,2cm
We will see that at least in the case of hypersurfaces the answer to these two questions is yes.

\begin{proof}[of Theorem \ref{theorem2}]  We have already seen that  $\widetilde{A} \subset \widetilde{A}^{R}$ for every R. Conversely, consider a function $f \in \bigcap_{R \in {\cal P}} \widetilde{A}^{R}$; does it belong to $\widetilde{A}$? 

Let us consider  the family of irreducible divisors in ${E}_{X}$ consisting of the $^{\tau} {D}_{{X}}^{\red}$ and of  the irreducible components of $\{{f}{\widearc{\otimes}} 1-1 {\widearc{\otimes}} {f}=0\}$. Let us denote by  ${^\tau}{\Delta_{f}}$  the set of    points of $^{\tau} {D}_{{X}}^{\red}$  which:
\begin{enumerate}
    \item do not belong to any other irreducible divisor of the family;
    \item are  smooth points of $^{\tau} {D}_{{X}}^{\red}$ and  of ${E}_{X}$.
\end{enumerate}

Since ${E}_{X}$ is normal, hence non-singular in codimension 1, ${^\tau}{\Delta_{f}}$ is a Zariski dense open set of  $^{\tau} {D}_{{X}}^{\red}$. At every point $w \in {^{\tau} {D}_{{X}}^{\red}}$, the local ideal of $^{\tau} {D}_{{X}}$ in ${E}_{X}$ is of the form $\left({s}^{\mu(\tau)}\right)$, where $s$ is a coordinate function of a local chart of ${E}_{X}$, and $\mu(\tau)$ an integer $\geq 1$ (the multiplicity  of the divisor ${^{\tau} {D}_{{X}}}$). Moreover, the function ${f} \widearc{\otimes} 1-1 \widearc{\otimes} {f}$ is of the form $us^{\nu(\tau)}$, where $u$ is a unit of the local ring of ${E}_{X}$ at the point $w$ and $\nu(\tau)$ is an integer $ \geq 0$.

Then,  it remains to prove that $v(\tau) \geq \mu(\tau)$ for every $\tau$ (see Section \ref{section2}). 

Let $S$ be the germ associated with a parametrization $R \in {\cal P}$ and let us denote by ${E}_{{X} / {S}}^{*}$ (resp. ${D}_{{X} / {S}}^{*}$)  the image of ${E}_{{X} / {S}}$ (resp. ${D}_{{X} / {S}}$)  in ${E}_{X}$  by the canonical map  ${E}_{{X} / {S}}   \xymatrix { \ar@{-->}[rr]& &} {E}_{X}$  defined at the beginning of the section. By definition, ${D}_{{X} / {S}}^{*} =  {E}_{{X} / {S}}^{*} \cap {D}_{X}$, so that if ${E}_{{X} / {S}}^{*}$ contains a point $w \in {^\tau}{\Delta_{f}}$, the divisor ${D}_{{X} / {S}}^{*}$ will be given in ${E}_{{X} / {S}}^{*}$, in a neighbourhood of this point, by the ideal $(s^{\mu(\tau)})$. If this ideal is not zero, i.e., if ${E}_{{X} / {S}}^{*}$ is not included in ${D}_{X}$, the relative Lipschitz condition:
$$({f} {\widearc{\otimes}} 1-1 {\widearc{\otimes}} {f} ){ | \ {D}_{X/S}} =0$$
implies that the function $({f} {\widearc{\otimes}} 1-1 {\widearc{\otimes}} {f}){| \ {E}_{{X} / {S}}^{*}}$ is divisible by $s^{\mu(\tau)}$ in a neighbourhood of $w$. 

By writing  ${f} {\widearc{\otimes}} 1-1 {\widearc{\otimes}} {f} = us^{\nu(\tau)}$ and  by remarking that $u$,  which is a unit of ${E}_{{X}}$, remains a unit after restriction to ${E}_{{X} / {S}}^{*}$, we deduce from this  that  $\nu(\tau) \geq \mu(\tau)$. \\

On the way, we had to admit that there exists an $R \in {\cal P}$ such that, for every non-trivial type $ \tau$, ${E}_{{X} / {S}}^{*}$ meets ${^\tau}{\Delta_{f}}$ and is not  locally included in ${^\tau}{\Delta_{f}}$. To make sure of this, and thus to complete the proof of Theorem \ref{theorem2}, it suffices to prove: 
\begin{lemma} \label{Lemme1}
For every Zariski  dense  open set ${^\tau}\Delta \subset {^{\tau} {D}_{{X}}}^{\red}$ ($\tau$ non-trivial) consisting of  smooth points of  $\ ^{\tau} {D}_{{X}}^{\red}$ which are  also  smooth points of  ${E}_{X}$, there exists a generic family of parametrizations $R$ for which the map ${E}_{{X} / {S}} \rightarrow {E}_{X}$ intersects ${^\tau}\Delta$ in at least one point $w$ and is an embedding transversal to ${^\tau}\Delta$ at this point.
\end{lemma}
(The condition of ``transversal embedding'' is obviously stronger than what we asked, but will be more manageable). \\

Let ${z}=(z_{1}, z_{2}, ..., z_{n})$ be a system of generators of the maximal ideal of ${\mathcal O}_{X,0}$, and denote by $ {^\tau} \tilde{z} : {^ \tau} {D}_{X}^{\red} \rightarrow \mathbf{P}^{N-1} $ the restriction of the morphism $ \tilde{z} : {E}_{X} \rightarrow \mathbf{P}^{N-1} $ of Section \ref{section2}. To every parametrization $R(a) = \mathbf{C}\left \{ {(az)_{1}, (az)_{2}, ..., (az)_{n}} \right \}$ defined by a matrix $a \in M_{N \times n} ({\mathbf C})$,  let us associate the $(N-n-1)$-plane $\mathbf{P}^{N-n-1}(a) \subset \mathbf{P}^{N-1}$ defined as the  projective  subspace associated to the kernel of the matrix $a$.

\begin{lemma} \label{Lemme2} If the map $ {^\tau} \tilde{z} : {^ \tau} \Delta \rightarrow \mathbf{P}^{N-1} $ is {\it effectively transversal} \footnote{The sentence: {\it the  map is transversal at $w$}  \index{transversal!map} expresses one of the following two possible cases: 
\begin{enumerate}
    \item $w$ sends itself outside of the subvariety into consideration;
    \item $w$ is sent into the subvariety  into consideration, and the image of the  tangent  map  to the point $w$ is a vector subspace transversal to the tangent space of this subvariety.
\end{enumerate} 
In the second case, we will say that the  map is {\it effectively transversal in $w$}. \index{transversal!effectively}} to $\mathbf{P}^{N-1}(a)$ at the point  $ w \in {^ \tau} \Delta $, then   the map $ {E} _ {X/S(a)} \rightarrow {E}_{X} $  is an embedding effectively transversal to ${^ \tau} \Delta$  at this point.  \end{lemma}

\begin{proof} [of Lemma \ref{Lemme2}] \ \ 
Since $ {^\tau} \tilde{z}$ is the restriction of $  \tilde{z} :  {E}_{X} \rightarrow \mathbf{P}^{N-1} $, the transversality of $ {^\tau} \tilde{z}$ implies the transversality of $ \tilde{z}$. 

Hence,  $ \tilde{z}^{-1} (\mathbf{P}^{N-n-1}(a))$ is a smooth subvariety  of ${E}_{X}$ of dimension $n$, which  intersects ${^\tau}\Delta$ transversely along the smooth subvariety  $  {}^{\tau}\tilde{z}^{-1} (\mathbf{P}^{N-n-1}(a))$ of  dimension $ n-1$. In particular,  $ \tilde{z}^{-1} (\mathbf{P}^{N-n-1}(a))$ is the closure of the complement of  $  {}^{\tau}\tilde{z}^{-1} (\mathbf{P}^{N-n-1}(a))$, i.e., the closure of its part located outside of the exceptional divisor. But outside of the exceptional divisor, the right vertical arrow of the  following commutative diagram is an isomorphism (since $\bar{{X}} \times \bar{{X}}$ is normal),  while the left  vertical arrow is surjective.

 $$\xymatrix {E_{X/S(a)}  \ \    \ar[d]    \ar[r]  &\ \ E_X \ar[d] \\
 \overline{X} \underset{S(a)}{\times}  \overline{X} \ \ \ar@{^{(}->}[r]& \ \   \overline{X} \times   \overline{X}
 }$$
 
\noindent
  Therefore, the image ${E}_{X/S(a)}^{*}$ of the upper arrow is identified with $ \overline{X} \underset{S(a)}{\times}  \overline{X}$, i.e., with $ \tilde{z}^{-1}(\mathbf{P}^{N-n-1}(a))$. Since the equality $$ {E}_{X/S(a)}^{*} = \tilde{z}^{-1}(\mathbf{P}^{N-n-1}(a))$$
is true outside of the exceptional divisor, it is true  everywhere,  by taking the  closure. 

It remains  to prove that ${E}_{X/S(a)}\rightarrow {E}_{X/S(a)}^{*} $ is an isomorphism (in a neighbourhood of $w$), but it is obvious. Indeed,  it is the germ of  a morphism between two smooth varieties of the same dimension  which sends a smooth divisor of one onto a smooth divisor of the other and which  is an isomorphism outside of these divisors. 
\end{proof}

By Lemma \ref{Lemme2}, we can consider Lemma  \ref{Lemme1} as a simple consequence of: 

\begin{lemma} \label{Lemme3}
There exists a dense open set of matrices $a \in M_{N \times n} (\mathbf C)$ for which the map $ {^\tau} \tilde{z} : {^ \tau} \Delta \rightarrow \mathbf{P}^{N-1} $ is effectively transversal to $\mathbf{P}^{N-n-1}(a)$ in at least one point $w \in {}^\tau \Delta$.
\end{lemma}
\begin{proof}
Let us construct a {\it stratification} \index{stratification} of ${^{\tau} {D}_{{X}}}^{\red}$ such that each of the following analytic sets is a union of strata: 

 \begin{enumerate}
\item  the  reduced fiber  $^{\tau} {D}^{\red}_X(0)$ of $^{\tau} {D}_{{X}}^{\red}$ over the origin $0 \in X$;
\item  the complement of the  Zariski open set ${^\tau}\Delta$. 
\end{enumerate}
 Let  us denote by ${W}$ be the maximal stratum of this stratification (obviously ${W} \subset {^\tau}\Delta$) and by $W_0$ the maximal stratum of one  (arbitrarily chosen) of the irreducible components of $^{\tau} {D}_{{X}}^{\red}$. We will assume that the stratification has been chosen  sufficiently fine  so that every  pair of strata  $({W}_0 , {V})$  satisfies Whitney  (a)-Condition  \cite{Whi65}, where ${V}$ belongs to the {\it star} \index{stratum!star of a}  of ${W}_0$ (see the appendix in the present paper). In these conditions, it follows from the appendix  that if the map $ {^\tau} \tilde{z} :{^{\tau}{D}^{\red}_{{X}} \rightarrow \mathbf{P}^{N-1}} $ has its restriction to ${W}_0$ effectively transversal to $\mathbf{P}^{N-n-1}(a)$ at a point ${w}_0 \in {W}_0$, then its restriction to ${W}$ will   be effectively  transversal to $\mathbf{P}^{N-n-1}(a)$ in at least one point ${w} \in {W}$ close to ${w}_0$.
\end{proof}
But, we will now prove the:
\begin{lemma} \label{Lemme4}
There exists a dense open set of matrices $a \in M_{N \times n} (\mathbf C)$ for which the map $ {}^\tau \tilde{{z}}_{\  | W_0} : {W}_0 \rightarrow \mathbf{P}^{N-1} $ is effectively transversal to $\mathbf{P}^{N-n-1}(a)$ in at least one point ${w}_0 \in {W}_0$.
\end{lemma}
\begin{proof}
Since $\tau$ is a non-trivial type, the image of the projection $^{\tau} {D}_{{X}}^{\red} \rightarrow {X} $ is of dimension $\leqslant n-1$, so that the dimension of the fiber $^{\tau} {D}^{\red}_X(0)$ must be at least $n$ (as dim$^{\tau} {D}_{{X}}^{\red} = 2n-1 $). Now, we know (Section \ref{section2}) that the map ${^\tau} \tilde{z}$ restricted to $^{\tau} {D}^{\red}_X(0)$ is a finite morphism. By considering   the algebraic variety of dimension $\geqslant n$ in $\mathbf{P}^{N-1}$ defined as the  image of a component of $^{\tau} {D}^{\red}_X(0)$, and the Zariski  dense  open set   of this variety defined as  the image of the set of points of ${W}_0$ where the morphism is a local isomorphism, we  see that Lemma \ref{Lemme4} is reduced to: 
\begin{lemma}   \label{Lemme5}
Consider  an algebraic variety of dimension $ \geqslant n$ in the projective space $\mathbf{P}^{N-1}$  and a  Zariski dense open set in this variety. The set of  $(N-n-1)$-planes of $\mathbf{P}^{N-1}$ which intersect transversely this open set  in at least one smooth point contains a dense open set of the Grassmann manifold.
\end{lemma}
The proof of this lemma is left to the reader. This completes the proof of Lemma \ref{Lemme4}.
 \end{proof}

\noindent
To summarize:
\[
\left.
\begin{array}{ccccr}
 \hbox{Lemma 3.7} &  \Longrightarrow &  \hbox{Lemma  3.6} &  \Longrightarrow  &   \hbox{Lemma 3.5}   \\
  &   &  &  &  \hbox{Lemma 3.4}
\end{array}
\right\}
\left.
\begin{array}{cccc}
&&& \\
  \Longrightarrow &  \hbox{Lemma  3.3} &  \Longrightarrow & \hbox{Theorem  \ref{theorem2}}
\end{array}
\right.
\]

\noindent
This completes the proof of  Theorem \ref{theorem2}. 
\end{proof}

\begin{remark}
The arguments of Section \ref{section2}  generalize without difficulty to the relative case. Thus,  for every analytic subalgebra $R\subset A$,  we have the notion of {\it  confluence locus  relative to $R$} and the relative analog of Theorem \ref{theorem 1}. If $R$ is a parametrization of $A$, we can see,  by an argument similar to that of Section $\ref{section2}$, that the dimension of the relative confluence  locus  admits the same lower bound $2n-N$ as in the  absolute case; in particular, the  confluence locus of a hypersurface $X$ relative to a parametrization are the codimension $1$ components of  the {\it relative singular locus} of $X$, i.e., the set of points of $X$ where the finite morphism $X\to S$ is not a submersion of smooth varieties.  
\end{remark}
We deduce from this: 

\begin{theorem}[relative version of Theorem \ref{theorem 1H}] \label{thm:1Hrelative}
Let $X\rightarrow S$ be a finite morphism of a complex analytic hypersurface to a smooth variety of the same dimension. Then, a meromorphic function on $X$ is locally Lipschitz relatively to $S$ at every point of $X$ if and only if  it is   locally Lipschitz relatively to $S$ at one point of each irreducible component (of codimension $1$) of its polar locus. 
\end{theorem}

\section{The particular case of plane curves}\label{sec: plans curves}  \label{section4}
 
Let $X\stackrel{(x,y)} \hookrightarrow \mathbf{C}^2$ be  a germ of reduced analytic plane curve and let  $\tilde{z}:E_X\to\mathbf{P}^1$ be the morphism   corresponding to the  germ of embedding $(x,y)$ (Section \ref{section2}). 

Let $U$ be the dense open set of $\mathbf{P}^1$ defined as the complement of the tangent  directions  of $X$.

Let $u\in U$. By performing a linear change of coordinates if necessary, we can assume that $u$ corresponds to the direction  of $\{x=0\}$.  In a neighbourhood of  $u$, we take as local coordinate $v$ in $\mathbf{P}^1$  the inverse of the slope in these coordinates.
\begin{proposition} \label{proposition1-IV}
In a neighbourhood of every point $w\in D_X \cap \tilde{z}^{-1}(u)$, $\tilde{z}{| D_X^{\red}}$ is an isomorphism, $E_X$ is smooth, and $E_X \cong E_{X/S(u)}\times D_X^{\red}$. 
\end{proposition}

\begin{proof}
Firstly, let us  remark that for every $|v|$ and $|t|$ (and obviously  every $|x|$ and $|y|$)  small enough, the line $x-vy=t$ remains non-tangent to $X$ and therefore, intersects $X$ transversally at simple points if $t$ is non-zero.   

Let $\Gamma \subset (\overline{X}\times \overline{X} - \overline{X} \underset{X}{\times} \overline{X} )\times \mathbf{P}^1$ be  the graph of the  map  defined in Section \ref{section2}. We  consider the  map $ \Psi_0 \colon \Gamma \rightarrow  \mathbf{C}\times \mathbf{P}^1$ defined by $(P,P',v) \mapsto (x(P)-vy(P),v)$ (by noticing that, by definition, $x(P)-vy(P)=x(P')-vy(P')$). The map $\Psi_0$ extends to a meromorphic  map $\widehat{E}_X \xrightarrow[]{\Psi_1} \mathbf{C}\times \mathbf{P}^1$ which is obviously bounded, and so extends locally to a unique morphism $E_X\xrightarrow[]{\Psi} \mathbf{C}\times \mathbf{P}^1$ (all this is done in a  neighbourhood of a point $w$ of $\tilde{z}^{-1}(u)$ on $E_X$).

It is easy  to check, and moreover it is  geometrically obvious, that $\Psi$ has finite fibers. In addition, by the remark of the beginning of the proof, it is clear that $\Psi$ is unramified outside of $\{0\} \times \mathbf{P}^1$. Therefore, the   ramification  locus is $\{0\} \times \mathbf{P}^1$ (unless it is empty).

Hence, the vector field  $\frac{\partial}{\partial v}$ of $\mathbf{C}\times \mathbf{P}^1$ is tangent to the ramification locus  of $\Psi$. Therefore, it lifts by $\Psi$ to  a holomorphic vector field  on the normal space $E_X$ (see \cite[Theorem 2]{Zar65-a}).\footnote{We can also see this by an argument similar to that of Lemma \ref{Lemma6} below.}  At every point  $w\in D_X \cap \tilde{z}^{-1}(u)$,   the integration of this vector field in a neighbourhood of $w$ endows locally  $E_X$ with a product structure $E_X \simeq \tilde{z}^{-1}(u)\times \Psi^{-1}(\{0\}\times \mathbf{P}^1)$.   

But, on the one hand, we can now apply Lemma \ref{Lemme2}  to prove that $\tilde{z}^{-1}(u) \simeq E_{X/S(u)}$ in a neighbourhood of $w$, and on other hand, again by  the above remark, $\tilde{z}^{-1}(u)$ does not meet any $^\tau D_X$ with trivial type $\tau$.

We conclude by noticing that since the origin, which is the only possible singularity of the germ $X$, is the support of all non trivial confluence loci ${}^{\tau}X$, we have: 
$$\Psi^{-1}(\{0\}\times \mathbf{P}^{1})=\bigcup_{\tau \hbox{ \small non  trivial}} {^\tau D_X}.$$
\end{proof}

\begin{corollary}(See Section  \ref{section2}) In this situation, the equation of $D_{X/S(u)}$ in $E_{X/S(u)}$ is the equation of $D_X$ in $E_X$.
\end{corollary}
We will now study the relative situation: 
  
  $$\xymatrix { X  \ \    \ar[rd]_{(x)}  \ar@{^{(}->}[r]^{(x,y)}&\ \ {\mathbf C^2}  \ar[d]^{pr_1} \\
 & \ \   S={\mathbf C}
 }$$
by assuming that $\{x=0\}$ is not tangent to $X$ at $0$.

We  will denote by $X_\alpha$ the irreducible components of $X$ and by $n_\alpha$ their multiplicities. 

For a local ring of dimension $1$, the normalized blow-up of an ideal is a regular ring which is nothing but the normalized ring. Hence, $E_{X/S}=\overline{\overline{X} \underset{S}{\times} \overline{X}}$. We can easily determine  the irreducible components of $E_{X/S}$ and the morphism $E_{X/S}\to S$ by using the following lemmas, after having noticed that an irreducible component of  $E_{X/S}$ projects onto  a pair  of irreducible components of $\overline{X}$.

\begin{lemma} \label{Lemma1-IV}
Set $m_{\alpha,\alpha'}= {\lcm}(n_\alpha , n_{\alpha'})$ and  let $\varphi :{\mathbf C}\{x\} \to \mathbf{C}\{s\}$ be given by $\varphi(x)=s^{m_{\alpha,\alpha'}}$. The set $B$ of  $\mathbf{C}\{x\}$-homomorphisms  
\[\mathbf{C}\{x^{1/n_\alpha}\} \underset{\mathbf{C}\{x\}}{\otimes}\mathbf{C}\{x^{1/n_{\alpha'}}\} \longrightarrow \mathbf{C}\{x\}\]
can be identified with the set of pairs $\{(\beta,\beta')\in \mathbf{C}^2  \colon  (\beta^{n_\alpha},\beta'^{n_{\alpha'}})=(1,1)\}$ by the correspondance:
 $$
 \left\{
\begin{array}{ccc}
x^{1/n_\alpha} \otimes 1 &\longmapsto  &\beta s^{m_{\alpha \alpha'}/n_{\alpha}} \\
1 \otimes x^{1/n_{\alpha'}} & \longmapsto & \beta' x^{m_{\alpha \alpha'}/n_{\alpha'}} \\
\end{array} \right.
$$
(the pairs ($\beta,\beta'$) correspond to the pairs of determinations of $(x^{1/n_\alpha},x^{1/n_{\alpha'}}))$.

If we endow $B$ with the equivalence relation: $b_1\sim b_2$ if $b_1 - b_2$ is  a $\mathbf{C}\{x\}$-automorphism of $\mathbf{C}\{s\}$ (it is the equivalence of pairs  of determinations ``modulo the monodromy''), then,  the set $B/_\sim$ has $(n_\alpha , n_{\alpha'})$ elements.
\end{lemma}
\begin{lemma}  \label{Lemma2-IV}
\[\overline{\mathbf{C}\{x^{1/n_\alpha}\}  \underset{{\mathbf{C}\{x\}}}{\otimes}\mathbf{C}\{x^{1/n_{\alpha'}}\}} =\underset{{B/_\sim}}{ \oplus}   \mathbf{C}\{x^{1/m_{\alpha \alpha'}}\} \]
with the obvious arrows.
\end{lemma}

Lemma \ref{Lemma2-IV} can be proved by using Lemma \ref{Lemma1-IV} and the universal property of  the normalization. The proof of  Lemma \ref{Lemma1-IV} is left to the reader.

We can now determine the equation of $D_{X/S}$ in $E_{X/S}$. At a point of an irreducible component of $E_{X/S}$, the ideal of $^\tau D_{X/S}$ is generated by $y\otimes1 - 1\otimes y= a_\tau s^{\mu(\tau)}$ (where $a_\tau$ is a unit of $\mathbf{C}\{s\}$), which can be interpreted as the difference of $y_{\alpha\beta}(x)-y_{\alpha'\beta'}(x)$ of the Puiseux expansions of $y_\alpha$ and $y_\alpha'$ computed for the ``determinations'' $(\beta, \beta')$ of $(x^{1/n_{\alpha}} , x^{1/n_{\alpha'}})$ corresponding to the chosen irreducible component: 
\[y_{\alpha\beta}(x)-y_{\alpha'\beta'}(x)=a_{\beta \beta'}x^{\mu(\beta, \beta')/m_{\alpha \alpha'}}\]
where $a_{\beta\beta'}$ is a unit of $\mathbf{C}\{x^{1/m_{\alpha\alpha'}}\}$,   $a_{\beta\beta'}=a_\tau$ and    $\mu(\beta,\beta')=\mu(\tau)$.

In the particular case where $X$ is irreducible of multiplicity $n$ at the origin, we deduce from this that the  sequence of the distinct $\mu(\tau)$ (for $\tau$ non trivial), indexed in increasing order, coincides with the sequence:
$$
\bigg\{ \frac{m_1}{n_1}n,\frac{m_2}{n_1 n_2}n,\dots, \frac{m_g}{n_1\dots n_g}n \bigg\}
$$
where the $\frac{m_i}{n_1\dots n_i}n$ are  the  characteristic Puiseux  exponents.

Now, we return  back to $E_X$ and $D_X$. If $^{\tau}D_{X/S}$ is an irreducible component of $D_X$, we know from Section  \ref{section2} that $\tilde{z}{|  {}^{\tau}D_{X}}$ is a finite morphism, and it follows  from Proposition \ref{proposition1-IV} that its ramification  locus is contained in the set of directions of tangent lines  to the irreducible components $X_{\alpha}$ and $X_{\alpha'}$ corresponding to $^{\tau} D_X$.
\begin{proposition}
\begin{enumerate}
    \item[(i)] If $X_\alpha$ and $X_{\alpha'}$ have the same tangent line, then $\deg \tilde{z}{| ^{\tau}D_{X}}=1$, so the number of types $\tau$ corresponding to the pair $(\alpha,\alpha')$ equals $(n_\alpha,n_{\alpha'})$.
     \item[(ii)] If $X_\alpha$ and $X_{\alpha'}$ have distinct tangent lines, then $\deg\tilde{z}{| {}^{\tau}D_{X}}=(n_\alpha,n_{\alpha'})$ and  there is a unique type $\tau$.
\end{enumerate}
\end{proposition} 

\begin{proof}
In  Case (i), let $r\in \mathbf{P}^1$ be  the direction of the common tangent line. Since $\mathbf{P}^1\setminus \{r\}$ is contractible, $^{\tau}D_{X} \setminus \tilde{z}^{-1}(r)$ is a trivial fiber bundle on $\mathbf{P} \setminus \{r\}$. This fiber bundle is connected since ${}^{\tau}D_X$ is irreducible, therefore, it is a covering space of degree $1$. 

Case (ii) is more delicate. Let $r_1$ and $r_2$ be  the two tangent directions and let $u\in \mathbf{P}^1 \setminus \{r_1,r_2\}$. We have to prove that we can join any two points of $E_{X/S(u)}$ by a path contained in $^{\tau}D_X$ and avoiding $\tilde{z}^{-1}(r_1)\cup\tilde{z}^{-1}(r_2) $. We can  do this by looking at two pairs of points $(P_\alpha, P_{\alpha'})$ and $(Q_\alpha, Q_{\alpha'})$, where  $P_\alpha, Q_{\alpha} \in X_\alpha \setminus \{0\}$ and $P_{\alpha'}, Q_{\alpha'}\in X_{\alpha'} \setminus \{0\}$ are  close to the origin and  located on the same line with slope $u$.  It is  possible to pass continuously from the pair $(P_\alpha, P_{\alpha'})$ to the pair  $(Q_\alpha, Q_{\alpha'})$ in such a way that the slopes of the lines joining the intermediate pairs stay at bounded distance from $r_1$ and $r_2$. We then conclude by taking the limit.\end{proof}

\section{Lipschitz saturation and Zariski saturation} \label{section5}
Let $R\subset A$ be a parametrization of a complex analytic algebra $A$, and let $X\to S$ be the  associated germ of  morphism of analytic spaces. Zariski defines a {\it domination} \index{domination} relation  between fractions of $A$  which, translated into transcendental   terms, can be formulated as follows: 
\begin{definition} {\it $f$ dominates $g$ over $R$} $(f \underset{R}{>}g)$ if and only if, for every pair $g_\beta(x), g_{\beta'}(x)$ of distinct determinations of $g$, considered as a multivalued function of $x \in S$, the quotient
\[\frac{f_\beta(x)-f_{\beta'}(x)}{g_\beta(x)-g_{\beta'}(x)}\]
 has bounded module, where $f_\beta(x)$ and $f_{\beta'}(x)$ denote the corresponding determinations of $f$.

An extension $B$ of $A$ in its total ring of fractions is said {\it saturated over $R$} if every fraction which dominates an element of $B$ belongs to $B$.

The  {\it saturated  algebra of $A$} (with respect to $R$) 
is defined as the smallest saturated algebra containing $A$.
\end{definition}
{\bf Question 3.} Is there a relation between the saturated algebra in the sense of Zariski and the algebra $\widetilde{A}^R$ defined in Section \ref{section3}? 
\vskip0,3cm
In the particular case of hypersurfaces, $A=R[y]$, we can easily see that  the Zariski saturation coincides with the set of fractions  which  dominate $y$, i.e., in this case, with the algebra $\widetilde{A}^R$ of Lipschitz fractions relative to the parametrization $R$.

In the  general case of an arbitrary codimension, $A=R[y_1,\dots,y_k]$, the   Zariski saturation and the Lipschitz saturation are  both more complicated to define, and  answering  Question 3 does not seem easy to us.

In some cases, including the case of hypersurfaces, Zariski can prove that his saturation is independant of the chosen parametrization as long as the latter is generic. Therefore, we obtain, in the case of hypersurfaces, a  positive answer to Questions 2 and 2' of Section \ref{section3}. More precisely, we have:

\begin{theorem} \label{theorem3}\footnote{(Added in 2020) For a more algebraic approach, see \cite{Li75a} and {Li75b}. For a more general result without the hypersurface assumption, see \cite{Bog74} and  \cite{Bog75}.}
Let $A$ be the complex analytic algebra of a hypersurface germ $X$, and consider the (generic) family $\mathcal{P}$ of the parametrizations defined by a direction of projection transversal to $X$ (i.e., not belonging to the tangent cone) at a generic point of each irreducible component of codimension $1$ of the singular locus. Then, for every $R\in \mathcal{P}$, $\widetilde{A}= \widetilde{A}^R$, which equals the Zariski saturation.
\end{theorem}
Indeed, this  family of parametrizations $\mathcal{P}$ is the one for which Zariski proves the invariance of his saturation (\cite[Theorem 8.2]{Zar68}).

\section{Equisaturation and Lipschitz equisingularity}  \label{section6} \index{saturation!Lipschitz}
 The notion  of saturation used in this section  is the Lipschitz saturation which, as we have just seen, coincides with the Zariski saturation in the case of hypersurfaces.

Let $r:X\to T$ be an analytic retraction of a reduced complex analytic space germ $X$ on a germ of  smooth subvariety $T \hookrightarrow X$.  Denote by $X_0= r^{-1}(0)$ the fiber of this retraction over the origin $0\in T$.
\begin{definition}\label{definition1} We say that $(X,r)$ is {\it equisaturated along $T$} \index{equisaturation} if the saturated germ $\widetilde{X}$ admits a product structure: $$\widetilde{X}=\widetilde{X}_0 \times T$$
compatible with the retraction $r$ (i.e., such that the second projection is $\widetilde{X}\rightarrow X \xrightarrow[]{r} T$).
\end{definition}

\begin{theorem} \label{theorem4}
If $(X,r)$ is equisaturated along $T$, then $(X,r)$ is topologically (and even Lipschitz) trivial along $T$. \index{triviality!topological}
\end{theorem}
By {\it topological triviality}, we mean the following property: for every embedding 
  $$\xymatrix { X  \ \   \ar[rd]_r  \ar@{^{(}->}[r] &\ \ {\mathbf C^N}  \ar[d]^{r^N} \\
 & T 
 }$$
of  the retraction $r$ in a retraction $r^N$  of a euclidean space, the pair $(\mathbf{C}^N,X)$ is homeomorphic to the product $(\mathbf{C}^{N-p}\times T, X_0\times T)$, in a compatible way with the retraction $r^N$.

By {\it Lipschitz triviality}, \index{triviality!Lipschitz} we mean that the above homeomorphism  is Lipschitz as well as its inverse.

\begin{proof}
Let $(t_1,\dots, t_p)$ be  a local coordinates system on $T$. By using the  product  structure  $\widetilde{X}=\widetilde{X_0} \times T$,  let us  denote by $\triangledown_i$ the vector field on $\widetilde{X}$ whose  first projection is zero and  whose  second one equals  $\frac{\partial}{\partial t_i}$. Let $A$ be the algebra of $X$;  $\triangledown_i$ is a derivation from $\widetilde{A}$ to $\widetilde{A}$. Then, by restriction, it defines a derivation from $A$ to $\widetilde{A}$. Let us  consider an embedding  $X  \hookrightarrow \mathbf{C}^N$, i.e.,  a system of $N$ generators of the maximal ideal of $A$: 
\[(z_1,z_2,\dots,z_{N-p},t_1\circ r , t_2\circ r,\dots,t_p\circ r)\]
(by a change of coordinates, all the systems can be reduced to this form). The functions $\triangledown_i z_1,\triangledown_i z_2\dots,\triangledown_i z_{N-p}$ are Lipschitz functions on $X$.  Then, they can extend to Lipschitz functions $g_{i,1}, g_{i,2}, \ldots, g_{i,N-p}  $ on all $\mathbf{C}^N$. Hence,   for all $i=1,2,\dots,p$,   we have  a Lipschitz vector field on $\mathbf{C}^N$:
\[g_{i,1}\frac{\partial}{\partial z_1}+g_{i,2}\frac{\partial}{\partial z_2}+\dots+g_{i,N-p}\frac{\partial}{\partial z_{N-p}}+\frac{\partial}{\partial t_i}\]
which is tangent to $X$ and    projects onto the vector field $\frac{\partial}{\partial t_i}$ of $T$.
Since they are  Lipschitz, these vector fields are locally integrable and their integration realizes the topological triviality of $X$.
\end{proof}

\section*{Relative equisaturation}
We will now define   a relative notion of equisaturation. \index{equisaturation!relative} Let $X/S$ be  a germ of analytic space  {\it relative to a parametrization},  consisting of the data of a reduced analytic germ $X$ of pure dimension $n$ and of  the germ of a finite morphism $X\to S$ on a germ of smooth variety of dimension $n$. Let 
\[r: X/S\rightarrow T\] 
be an {\it analytic retraction of the relative analytic space $X/S$ on a smooth subvariety $T$}.  By this, we mean the datum of a commutative diagram: 
 $$\xymatrix {  &&\\
 T  \ \    \ar@/^2pc/[rr]^{identity}  \ar@{^{(}->}[r]  \ar@{^{(}->}[rd] & X   \ar[d]  \ar[r]^r &T  \\
 &  S \ar[ru]  &  
 }$$
 Denote by $X_0/S_0$  the relative analytic space defined as the  inverse image of  the point $0\in T$ by this retraction.
\begin{definition}[Relative Definition \ref{definition1}] We say that $(X/S,r)$ is {\it equisaturated  along $T$}, if the germ of relative saturated  space $\widetilde{X}^S/S$ admits a  product structure: 

 $$\xymatrix { \widetilde{X}^{S} \ar[d]   &=& \widetilde{X_0}^{S_0} \ar[d] &\times& T \ar[d]^{id}\\
 S&=&S_0&\times & T
 }$$
which is compatible with the retraction $r$.
\end{definition}
In the case where $X$ is a hypersurface, it results  immediately from Theorem  \ref{theorem3} that if $S$ is a generic parametrization,  the equisaturation of $X/S$ ({\it relative equisaturation})  implies the equisaturation of $X$ ({\it absolute equisaturation}).
\vskip0,3cm\noindent
{\bf Question 4.} Conversely, does the equisaturation of $X$ imply the existence of a generic parametrization $S$ such that $X/S$ is equisaturated? 
\\ 

It would be interesting to know the answer  to this question because  the  work of Zariski gives a lot of informations on the relative notion of equisaturation. 

We assume in the sequel that $X$ is a hypersurface. Let $R=\mathbf{C}\{z_1,\dots,z_n\}$ be a parametrization of $A$.
We can write: 
\[A=R[y]=R[Y]/(f),\] 
where $f$ is a   reduced monic polynomial in $Y$ with coefficients in $R$ and where  $y= Y+(f)$ is the residue class of $Y$ modulo $f$. The reduced discriminant of this polynomial generates  an ideal in $R$ which  depends only of $A$ and $R$; we will call it  the {\it ramification ideal}   \index{ramification!ideal} of the  parametrization $R$. We will denote by   $\Sigma \subset S$ the subspace defined by this ideal;  this subspace will be called the {\it  ramification locus} \index{ramification!locus} of $X/S$.
\begin{definition}
We say that $(X/S,r)$ has  {\it trivial ramification locus  along $T$} if the pair $(S,\Sigma)$ admits a product structure:
 $$\xymatrix {\Sigma \ar@{^{(}->}[d]   &=& \Sigma_0 \ar@{^{(}->}[d] &\times& T \ar[d]^{id}\\
 S&=&S_0&\times & T
 }$$
which is compatible with the retraction $r$.
\end{definition}

\begin{theorem}[Zariski  \cite{Zar68} ]  \label{theorem5}
Let $X$ be a hypersurface. The following two properties are equivalent:
\begin{enumerate}
    \item[(i)] $(X/S,r)$ is equisaturated along $T$;
    \item[(ii)] $(X/S,r)$ has trivial ramification locus.
    \end{enumerate}
Moreover,  these two properties imply the topological triviality along $T$ of the hypersurface $X$. 

\end{theorem} 
Notice that in the case of a generic parametrization, where the relative equisaturation implies the absolute equisaturation, the last  part  of Theorem \ref{theorem5} is a simple Corollary of  our Theorem \ref{theorem3}. But Zariski proves Theorem \ref{theorem5}  for any parametrization. 

We will limit ourselves to the proof of  the implication $(ii) \Rightarrow (i)$ and we refer  to   \cite{Zar68} for the rest. 
\begin{lemma} \label{Lemma6}
Every derivation of the ring $R$ in itself which leaves stable the ideal of ramification  extends   canonically into a derivation of the relative saturation $\widetilde{A}^R$ in itself.
\end{lemma} 

\begin{proof} 
Since $A$ is finite over $R$, every derivation 
$\triangledown \colon R\to R$ admits a canonical extension to the ring of fractions of $A$. Explicitly, we have: 
\[\triangledown y=-\big(\sum_{i=1}^{n}\frac{\partial f}{\partial z_i}
\triangledown  z_i \big)/\frac{\partial f}{\partial y}.\]
We have to prove  that under the  hypothesis of the lemma, $\triangledown g \in \widetilde{A}^R$ for every $g\in \widetilde{A}^R$. But the polar locus of every $g \in \overline{A}$ is obviously included in the singular locus  of $X$, so, in the zero locus of $\frac{\partial f}{\partial y}$. By writing
\[\triangledown g=\frac{\partial g}{\partial y}\triangledown y+\sum_{i=1}^{n}\frac{\partial g}{\partial z_i}\triangledown z_i,\]
we deduce from this   that the polar locus of $\triangledown g$ is included in the zero locus of $\frac{\partial f}{\partial y}$.

In order to check that $g\in \widetilde{A}^R$,  it is then sufficient (by Theorem \ref{thm:1Hrelative}) to check it at a generic point of each irreducible component (of codimension 1, of course) of the zero locus   of $\frac{\partial f}{\partial y}$.

Let $^S X$ be such an irreducible component, restricted to a small neighbourhood of one of its points.  For a generic choice of  the point, we can assume that: 
\begin{enumerate}
\item [(1)]  $^S X$ is smooth and the restriction to $^S X$ of  the morphism: $X\to S$ is an embedding; 
    \item [(2)] $^S X=(X|\Sigma)^{\red}$, where $\Sigma \subset S$ denotes the image of $^S X$, i.e., the ramification locus  of the morphism $X\to S$;
    \item [(3)]  the finite cover $D_{X/S}^{\red} \to\  ^SX$ is \'etale, i.e.,  $D_{X/S}^{\red}$ is a  disjoint union of\footnote{(added in 2020) \dots open subsets of components $^\tau D_{X/S}^{\red}$, isomorphic to their image in \dots} components $^\tau D_{X/S}^{\red}$ isomorphic to $^S X$.
\end{enumerate}

By  (1), $\Sigma$ is a smooth divisor of $S$ and we can choose local coordinates   $(x,t_1,t_2,\dots,t_{n-1})$   in $S$ so that $x=0$ is a local equation of this  divisor.  (2) means that $x$ does not vanish outside of $^S X $.

Locally in $S$, the submodule of the derivations which leave stable  the ramification ideal $(x)$ is generated by $x\frac{\partial}{\partial x}$ and the  $\frac{\partial}{\partial t_i}$'s.

Therefore, it suffices to prove that the functions $x\frac{\partial g}{\partial x}$ and $\frac{\partial g}{\partial t_i}$ are  Lipschitz  relatively to $S$.

Consider   the space $E_{X/S}$, that we can assume to be smooth,  in a  neighbourhood of one of the components $^\tau D_{X/S}^{\red}$ of the \'etale cover of (3). The function $x$ is well defined on $E_{X/S}$ (by composition with the canonical morphism $E_{X/S}\to \overline{X} \underset{S}{\times} \overline{X} \to S$) and by (2), it does not vanish outside of $D_{X/S}^{\red}$.
Therefore it is of the form $x= s^{m(\tau)}$, where $s=0$ is an irreducible equation of the smooth divisor $^\tau D_{X/S}^{\red}$.
On the other hand, the ideal of the non-reduced divisor $^\tau D_{X/S}$ is generated by: 
\[^\tau\Delta y=y\otimes 1 - 1 \otimes y = a_{\tau}(t) s^{\mu(\tau)}+\cdots, \]
where $a_\tau$ must be a unit of the ring $\mathbf{C}\{t_1,t_2,\dots,t_{n-1} \}$ since $^\tau {\Delta y}$ vanishes only  on $^\tau D_X^{\red}$.

Thus, for every $\tau$, we have a series expansion whose terms are   increasing  powers of   $x^{1/m(\tau)}$ (compare to Section \ref{sec: plans curves}):
\[^\tau {\Delta y}=a_{\tau} (t) x^{\frac{\mu(\tau)}{m(\tau)}}+\cdots,\]
and a function $g$ will be Lipschitz relatively to $S$ if and only if for every $\tau$, the series expansion of $^\tau \Delta g$ into rational powers of $x$ has no terms with exponents less than $\mu(\tau)/m(\tau)$. Let $g$ be such  a function: 
\[^\tau {\Delta g}=b_{\tau} (t) x^{\frac{\mu(\tau)}{m(\tau)}}+\cdots\]
We have: 
\[^\tau \Delta(x\frac{\partial g}{\partial x})=x \frac{\partial}{\partial x}(^\tau \Delta g)=x\bigg( \frac{\mu(\tau)}{m(\tau)}b_\tau (t) x^{\frac{\mu(\tau)}{m(\tau)}-1}+\cdots \bigg) \]
and: 
\[^\tau \Delta(\frac{\partial g}{\partial t_i})= \frac{\partial}{\partial t_i}(^\tau \Delta g)=\frac{\partial b_\tau (t)}{\partial t_i}x^{\frac{\mu(\tau)}{m(\tau)}}+\cdots,\]
so that $x\frac{\partial g}{\partial x}$ and $\frac{\partial g}{\partial t_i}$ are still functions of the same type, i.e.,  Lipschitz functions relative to $S$. This completes the proof of Lemma \ref{Lemma6}
\end{proof}

\begin{proof}{\bf of $\mathbf{(ii)\Rightarrow (i)}$ of theorem \ref{theorem5}.}

\noindent
Let us choose  local coordinates  $(x_1,x_2, \dots,x_{n-p}, t_1,t_2, \dots,t_p)$ in $S$  compatible with the   product structure $S_0 \times T$. The vector field $\frac{\partial}{\partial t_i}$ is tangent to the ramification locus  $\Sigma=\Sigma_0 \times T$. Therefore,  by  Lemma \ref{Lemma6}, it lifts to a  holomorphic vector field $\triangledown_i$ on $\widetilde{X}^S$. The integration of these $p$ vector fields $\triangledown_1,\dots,\triangledown_p$ realizes the  desired product structure  on $\widetilde{X}^S$ .
\end{proof}

\section*{Speculation on equisingularity} \index{equisingularity}

We would like to find a "good" definition of equisingularity of $X$ along $T$, satisfying if possible  the two  following properties: 
\begin{description}
\item [(TT)] the equisingularity implies the topological triviality; \index{triviality!topological}
\item [(OZ)] the set of points of $T$ where $X$ is equisingular forms a dense Zariski open set.
\end{description}

Equisaturation satisfies (TT) (Theorem \ref{theorem3} above), but satisfies (OZ) only in the case where $\hbox{codim}_XT=1$ (equisaturation of a family of curves coincides with equisingularity). In the general case, one can find some $X \to T$ such that $X$ is not equisaturated at any point of  $T$\footnote{Here we are thinking  about the relative  equisaturation characterized (Theorem \ref{theorem5}) by the triviality of the  ramification locus. But likely, the notion of absolute  equisaturation  leads to about the same thing - don't we want to answer yes to  Question 4? \par\noindent Added in 2020: The approaches of E. B\"oger in \cite{Bog75} and J. Lipman in \cite{Li75b} would probably lead to a positive answer.}.

Zariski proposed  a definition to the equisingularity of hypersurfaces that generalizes the idea of trivialization  of the ramification locus  (\cite{Zar37}, \cite{Zar64}):  the hypersurface $X$ is {\it equisingular} \index{equisingularity!Zariski} along $T$ if, for a generic parametrization, the ramification locus $\Sigma$ is equisingular along (the projection of) $T$. Since  the codimension of $T$ in $\Sigma$ is smaller than  its codimension   in $X$ minus one, we therefore obtain  a definition of the equisingularity by induction on the codimension.\footnote{(Added in 2020) This theory was described by Zariski in \cite{Zar79}, \cite {Zar80}. The reason why equisaturation does not satisfy (OZ) in general is that it corresponds to a condition of analytical triviality of the discriminant, which of course does not satisfy (OZ) in general. See also \cite{LiT79}.}

This definition satisfies (OZ), but we do not know how to prove (TT)\footnote{(added in 2020) There are now several results where Zariski equisingularity implies topological triviality sometimes via the Whitney conditions. See \cite{Var73}, \cite{Spe75}.}. 

In the case where $T$ coincides with the singular locus of $X$ (family of analytic spaces with isolated singularities), Hironaka found a criterion of equisingularity which satisfies (TT) and (OZ) at the same time. This criterion is defined [Hir64] in terms of the normalized blow-up of an ideal (i.e., the product of the ideal of $T$ by the Jacobian ideal of $X$). The topological triviality is proved  by integrating a vector field \footnote{Cf. H. Hironaka (not published but see \cite{Hir64-b}).}, but: 
\begin{enumerate}
    \item instead of being holomorphic on $\overline{X}$, this vector field is differentiable (i.e., $ {C}^{\infty}$) on the blown-up space  $\widehat X$ of $X$ (the normalized blow-up of the ideal mentioned above).
    \item instead of being Lipschitz on $X$, i.e., satisfying a Lipschitz inequality  for every ordered pair of points in $X \times X$, this vector field   satisfies  a Lipschitz inequality only for the ordered pair of points in $T \times X$. 
\end{enumerate}

The general solution to the problem of equisingularity will maybe use some rings of this type of functions ($ {C}^{\infty}$ in a blown-up space and ``weakly Lipschitz" below).\footnote{(Added in 2020) The idea of considering vector fields which are differentiable on some blown-up space was used by Pham in \cite{Pha71a} and, in real analytic geometry by Kuo who introduced \textit{blow-analytic} equivalence of singularities; see \cite{Kuo85}.}

\section*{Appendix: stratification, Whitney's $(a)$-property and transversality}
 A {\it stratification} \index{stratification}\footnote{See also David Trotman's article "Stratifications, Equisingularity and Triangulation" in this volume.} of an analytic (reduced) space  $X$ is a locally finite partition of $X$ in smooth varieties called {\it strata}, such that: 
\begin{enumerate}
    \item the closure $\overline{W}$ of every stratum $W$ is an (irreducible) analytic space;
    \item the boundary $\partial W= \overline{W}\setminus W$ of every stratum $W$ is a union of strata.
\end{enumerate}

We call {\it star} of a stratum \index{stratum!star of a} $W$ the set of strata which have $W$ in their boundary. 

Let $(W_0,W)$ be  an ordered pair of strata, with $W_0 \subset  \partial W$. We say that this ordered pair satisfies the {\it property (a) of Whitney} \index{Whitney!(a)-property} at a point $x_0 \in W_0$ if for every sequence of points $x_i \in W$ tending to  $x_0$ in such a way that the tangent space $T_{x_i}(W)$ admits a limit, this limit contains the tangent space $T_{x_0}(W)$ (we suppose that $X$ is locally embedded in a Euclidean space, in such a way that   the tangent spaces are realized as subspaces of the same vector space; the property (a) of Whitney is independent of the chosen embedding). For every ordered pair of strata $(W_0,W)$ of a stratification, there exists a Zariski dense open set  of points of $W_0$ where the property (a) of Whitney is  satisfied \cite{Whi65}. We can then refine every stratification into a stratification such that the property (a) of Whitney is satisfied at every point for every ordered pair of strata. 
\begin{proposition} \label{proposition finale}
Let $(X,x_0)$ be a   stratified germ of complex analytic space such that the ordered pairs of strata $(W_0,W)$ satisfy the property (a) of Whitney, where $W_0$ denotes the stratum which contains $x_0$ and where $W$ is any stratum of the star of  $W_0$. Let $\varphi: X \to \mathbf{C}^m$ be a morphism germ  such that $\varphi{| W_0}$ is effectively transversal to the value  $0$ at the point $x_0$. Then, for every stratum $W$, $\varphi{| W}$ is effectively transversal to the   value $0$ at (at least) one point of $W$ arbitrarily close to $x_0$.
\end{proposition}
\begin{proof}
The transversality of $\varphi{| W}$ at every point close to $x_0$ is an obvious consequence of the property (a) of Whitney for the ordered pair $(W_0,W)$. It  remains to prove the effective transversality  i.e., to prove  that $(\varphi{| W })^{-1}(0)$ is not empty. But $(\varphi{| \overline{W}} )^{-1}(0)$ is a close analytic subset of $\overline{W}$, non empty (because it contains the point $x_0$) and defined by $m$ equations. Therefore its codimension is at most $ m$. If ${(\varphi{| W} )}^{-1}(0)$ were empty, then $\partial W$ would contain at least one stratum $W'$ such that ${(\varphi{| W'} )}^{-1}(0)$ is non empty and of  dimension $\geqslant \dim W - m$. But on the other hand,  the transversality of $\varphi{| W'}$ implies that ${(\varphi{| W'} )}^{-1}(0)$ is a smooth variety of dimension $ < \dim W' - m $, and then   of dimension $< \dim W - m $. We then  get a  contradiction. 
\end{proof}

\begin{remark}
Of course,  in  the statement  of  Proposition   \ref{proposition finale}, we could replace the transversality relative to the value of $0$ by the transversality relative to a smooth variety of $\mathbf{C}^m$.
\end{remark}


\begin{thebibliography}{Whi65}

\bibitem[Abh64]{Abh64}
Shreeram~Shankar Abhyankar,
\newblock {\em Local analytic geometry},
\newblock Pure and Applied Mathematics, Vol. XIV. Academic Press, New York-London, 1964.

\bibitem[Bou61]{Bou61}
Nicolas Bourbaki,
\newblock {\em Alg\`ebre Commutative, Chapitre II},
\newblock Hermann, Paris 1961.
  
\bibitem[Hir64a]{Hir64-a}
Heisuke Hironaka,
\newblock Resolution of singularities of an algebraic variety over a field of
  characteristic zero {I}, {II},
\newblock {\em Ann. of Math. (2) {\bf 79} (1964), 109--203; ibid. (2)},
  79:205--326, 1964.
  
\bibitem[Hir64b]{Hir64-b}
Heisuke Hironaka,
\newblock Equivalences and deformations of isolated singularities,
\newblock {\em Woods Hole A.M.S. Summer Institute on algebraic geometry}, 1964. \\Available at https://www.jmilne.org/math/Documents/

\bibitem[Lip69]{Lip69}
Joseph Lipman,
\newblock Rational singularities, with applications to algebraic surfaces and
  unique factorization,
\newblock {\em Inst. Hautes \'{E}tudes Sci. Publ. Math.}, 36:195--279, 1969.


\bibitem[Whi65]{Whi65}
Hassler Whitney,
\newblock Tangents to an analytic variety,
\newblock {\em Ann. of Math. (2)}, 81:496--549, 1965.

\bibitem[Zar37]{Zar37} Oscar Zariski,
\newblock A theorem on the {P}oincar\'{e} group of an algebraic hypersurface,
\newblock {\em Ann. of Math. (2)}, 38(1):131--141, 1937.

\bibitem [Zar64]{Zar64}
Oscar Zariski,
\newblock Equisingularity and related questions of classification of singularities,
\newblock {\em Woods Hole A.M.S. Summer Institute on algebraic geometry}, 1964. \\Available at https://www.jmilne.org/math/Documents/

\bibitem[Zar65a]{Zar65-a} Oscar Zariski,
\newblock Studies in equisingularity, {I}, {E}quivalent singularities of plane
  algebroid curves,
\newblock {\em Amer. J. Math.}, 87:507--536, 1965.

\bibitem[Zar65b]{Zar65-b}
Oscar Zariski,
\newblock Studies in equisingularity, {II}, {E}quisingularity in
codimension $1$ (and characteristic zero),
\newblock {\em Amer. J. Math.}, 87:972--1006, 1965.

\bibitem[Zar68]{Zar68}
Oscar Zariski,
\newblock Studies in equisingularity, {III}, {S}aturation of local rings and
  equisingularity,
\newblock {\em Amer. J. Math.}, 90:961--1023, 1968.
\par\medskip
\centerline{A SAMPLE OF MORE RECENT BIBLIOGRAPHY}\par\medskip

\bibitem[Bog74]{Bog74} Erwin B\"oger,
\newblock Zur Theorie der Saturation bei analytischen Algebren,
\newblock {\em Math. Ann.}, 211:119--143, 1974.

\bibitem[Bog75]{Bog75} Erwin B\"oger,
\newblock \"Uber die Gleicheit von absoluter und relativer Lipschitz-Saturation bei analytischen Algebren,
\newblock {\em Manuscripta Math}, 16, 229--249, 1975.

\bibitem[Fer03]{Fer03} Alexandre Fernandes,
\newblock Topological equivalence of complex curves and Bi-Lipschitz Homeomorphisms,
\newblock {\em Michigan Math. J.}, 51: 593--606, 2003.

\bibitem[Gaf10]{Gaf10} Terence Gaffney,
\newblock  Bi-Lipschitz equivalence, integral closure and invariants,
\newblock in {\em Real and Complex singularities}, M. Manoel, M. C. Romero Fuster, C. T. C. Wall, editors, London Math. Soc. Lecture Notes Series, no. 380. Cambridge U. P., 125--137, 2010.

\bibitem[Kuo85]{Kuo85}  Tzee Char Kuo,
\newblock On classification of real singularities,
\newblock {\em Invent. Math.}, 82(2): 257--262, 1985.

\bibitem[Li75a]{Li75a} Joseph Lipman,
\newblock Absolute saturation of one-dimensional local rings,
\newblock{\em Amer. J. Math.} 97(3):771--790, 1975. 

\bibitem[Li75b]{Li75b} Joseph Lipman,
\newblock Relative Lipschitz saturation,
\newblock{\em Amer. J. Math.} 97:791--813, 1975. 

\bibitem[Li76]{Li76} Joseph Lipman,
\newblock Errata: Relative Lipschitz saturation,
\newblock{\em Amer. J. Math.} 98(2):571, 1976. 

\bibitem[LiT79]{LiT79} Joseph Lipman and Bernard Teissier,
\newblock Introduction to Volume IV, \textit{Equisingularity on algebraic varieties}, of Zariski's Collected Papers,
\newblock MIT Press, Boston, 1979.

\bibitem[NeP14]{NeP14} Walter D. Neumann and Anne Pichon,
\newblock Lipschitz geometry of complex curves,
\newblock {\em J. Singul.}, 3:225--234, 2014.

\bibitem[Pha71a]{Pha71a} Fr\'ed\'eric Pham,
\newblock D\'eformations \'equisinguli\`eres des id\'eaux Jacobiens de courbes planes.
\newblock {\em Proceedings of Liverpool Singularities Symposium, II (1969/70)},
\newblock Lecture Notes in Math., Vol. 209, Springer, 218--233, 1971.

\bibitem[Pha71b]{Pha71b} Fr\'ed\'eric Pham,
\newblock Fractions lipschitziennes et saturation de Zariski des alg\`ebres analytiques complexes (expos\'e d'un travail fait avec Bernard Teissier),
\newblock {\em Actes du Congr\`es International des Math\'ematiciens, Nice 1970}, Gauthier-Villars, Paris, Volume II, 649--654,  1971.

\bibitem[Spe75]{Spe75}Jean-Paul Speder,
\newblock Equisingularit\'e et conditions de Whitney, 
\newblock {\em Amer. J. Math.}, 97 (3); 571-- 588, 1975.

\bibitem[Var73]{Var73} Alexander Varchenko,
\newblock The connection between the topological and the algebraic-geometric
equisingularity in the sense of Zariski,
\newblock {\em Funkcional Anal. i Priloz.}, 7 (2) 1-- 5, 1973.

\bibitem[Zar71a]{Zar71a}
\newblock  General theory of saturation and of saturated local rings, I,
\newblock {\em Amer. J. Math.}, 93 (3); 573-- 648, 1971.

\bibitem[Zar71b]{Zar71b}
\newblock  General theory of saturation and of saturated local rings, II,
\newblock {\em Amer. J. Math.}, 93 (4); 872-- 964, 1971.

\bibitem[Zar73]{Zar73}
\newblock  Quatre expos\'es sur la saturation (Notes de J-J. Risler),
\newblock in: {\em Singularit\'es \`a Carg\`ese (Rencontre Singularit\'es en G\'eom\'etrie Analytique, Institut d'Etudes Scientifiques de Carg\`ese, 1972)}, Asterisque, n$^\circ$ 7 et 8, Soc. Math. France, Paris, 1973, 21--39.

\bibitem[Zar75]{Zar75}
\newblock  General theory of saturation and of saturated local rings, III,
\newblock {\em Amer. J. Math.}, 97 (2); 415-- 502, 1975.

\bibitem[Zar79]{Zar79}Oscar Zariski, 
\newblock  Foundations of a general theory of equisingularity on r-dimensional algebroid and algebraic varieties, of embedding dimension r+1, 
\newblock {\em Amer. J. Math.}, 101(2): 453--514, 1979.

\bibitem[Zar80]{Zar80} Oscar Zariski,
\newblock Addendum to my paper: "Foundations of a general theory of equisingularity on r-dimensional algebroid and algebraic varieties, of embedding dimension r+1'' [Amer. J. Math. 101 (1979), no. 2, 453--514],
\newblock {\em Amer. J. Math.}, 102(3): 453--514, 1980.
\end{thebibliography}
\bibliographystyle{alpha}

\printindex
\end{document}